\documentclass[a4 paper,11pt]{amsart}

\usepackage{amsmath} %
\usepackage{amssymb} %
\usepackage{amscd} %
\usepackage{amsthm} %
\usepackage{mathrsfs} %
\usepackage{euscript}
\usepackage{eufrak}
\usepackage{ulem}
\usepackage[alphabetic]{amsrefs}
\usepackage{ascmac} 
\usepackage{comment}
\usepackage{bm}
\usepackage[english]{babel}
\usepackage[dvipdfmx]{graphicx} 
\usepackage{hyperref}
\usepackage{xcolor}
\definecolor{unbleu}{rgb}{0.03, 0.15, 0.4}
\hypersetup{
pdfborder = {0 0 0},
colorlinks,
linkcolor=unbleu,
citecolor=unbleu,
urlcolor=unbleu
}

\usepackage{enumerate}
\usepackage{fancyhdr}
\usepackage{tikz}
\usepackage{here}
\usetikzlibrary{calc}
\usepackage{lineno}
\usepackage{braids}
 \newtheorem{theorem}{Theorem}[section]
 
 \newtheorem{proposition}[theorem]{Proposition}

\newtheorem{corollary}[theorem]{Corollary}

\theoremstyle{definition}
\newtheorem{definition}[theorem]{Definition}
\newtheorem{remark}[theorem]{Remark}

\newcommand{\C}{\mathbb C}
\newcommand{\R}{\mathbb R}
\newcommand{\Z}{\mathbb Z}
\newcommand{\N}{\mathbb N}

\begin{document}

\title[]{Hamiltonian systems and monotone twist mappings for braids}

\author[Y. Kajihara]{Yuika Kajihara}
\address{Department of Mathematics, Kyoto University, Kitashirakawa Oiwake-cho, Sakyo-ku,
Kyoto, 606-8502, Japan}
\email{kajihara.yuika.6f@kyoto-u.ac.jp}

\author[M. Shibayama]{Mitsuru Shibayama}
\address{Department of Applied Mathematics and Physics, 
Graduate School of Informatics, Kyoto  University, Yoshida-Honmachi, Sakyo-ku, Kyoto 606-8501, Japan}
\email{shibayama@amp.i.kyoto-u.ac.jp}

 \date{\today}
\subjclass[2010]{\textcolor{black}{Primary} }

\begin{abstract}
In 1986, Moser showed that for a given area-preserving map, there exists a Hamiltonian system that realizes it on the Poincaré section.  
Using his technique, we show that for any braid, there exists a Hamiltonian system whose orbits realize the given braid.  
In particular, when the braid is pseudo-Anosov, so is the Poincar\'e map of the corresponding Hamiltonian. 
\end{abstract}
\maketitle

\section{Introduction}
Consider a discrete dynamical system defined on a two-dimensional region bounded by two invariant curves, induced by a \textit{monotone twist mapping} on an annulus (this name comes from Mather \cite{Mather82}), defined as follows.
\begin{definition}[Monotone twist mappings]
\label{defi:mtm}
Set $a,b \in \R$ with $a<b$. We call
\[
\varphi \colon (x,y) \to (x_1,y_1) (= (f(x,y),g(x,y)))
\]
a monotone twist mapping on an annulus
\[
A=\{(x,y) \in \mathbb{R}^2 \mid a \le y \le b \}
\]
if $f, g \in C^{\infty}(A)$ and $\varphi$ satisfy the following properties.
\begin{enumerate}
    \item $\partial(f,g)/\partial(x,y) \equiv 1$
    \item $f(x+1,y)=f(x,y)+1, g(x+1,y)=g(x,y)$
    \item $g(x,y)=y$ for $y=a,b$
    \item $\partial f/\partial y>0$
\end{enumerate}
\end{definition}
Although the set $A$ is not precisely an annulus, we can regard $f,g$ as the lifts for a map on $\R / \Z \times [a,b]$ by the condition $(2)$ in Definition \ref{defi:mtm}.
It is known, through the work of Moser \cite{Moser86}, that a trajectory of this system can be realized on a Poincar\'{e} section of periodic solutions of a Hamiltonian system.
His main theorem is
\begin{theorem}[\cite{Moser86}]
\label{theorem:Moser}
For a given $C^\infty$-monotone twist mapping $\varphi$, there exists a Hamiltonian function $H=H(t,x,y) \in {C}^{\infty}(\R \times A)$ with the following properties.
\begin{enumerate}
\item It satisfies the following three conditions.
\begin{enumerate}
\item $H(t+1,x,y)=H(t,x,y)=H(t,x+1,y)$
\item $H_x(t,x,y)=0$ for $y=a,b$
\item $H_{yy}>0$
\end{enumerate}
\item The flow $\varphi_t$ taking the initial values $(x_0,y_0)$ of the solutions of
\[
\dot{x} = H_y(t,x,y), \ \dot{y}=- H_x(t,x,y)
\]
into the values $(x(t),y(t))$ agrees with the given mapping $\varphi$ for $t=1$, i.e.,
$\varphi(x(0),y(0))=(x(1),y(1))$.
\end{enumerate}
\end{theorem}
The aim of this paper is to construct a Hamiltonian that realizes the given braids, applying Moser's work.
To state our main theorem, we only need an extremely basic definition and properties of braids.
Let $B_n$ be the braid group on $n$ strands, i.e.,
\[
B_n=\left\langle \sigma_1, \dots, \sigma_{n-1}
\left|
\begin{matrix}
\ \sigma_i \sigma_j= \sigma_j \sigma_i \hspace{12mm} \mbox{if}\ |i-j|>1 \\
\ \sigma_i \sigma_j \sigma_i = \sigma_j \sigma_i \sigma_j \hspace{5mm} 
 \mbox{if}\ |i-j|=1
\end{matrix}
\right.
\right\rangle
\]
where $\sigma_i$ and products represent geometric objects, as shown in Figure \ref{fig:3braids}.
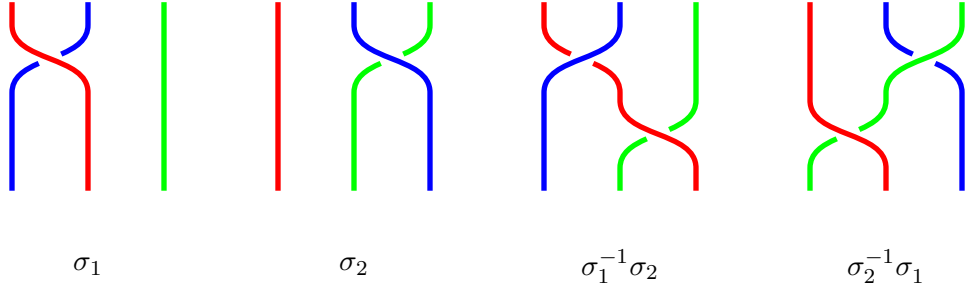
\begin{figure}[tbp]
    \centering
   \begin{tikzpicture}[
  /pgf/braid/.cd,
  line width=2pt,
  style strands={1}{red},
  style strands={2}{blue},
  style strands={3}{green},
  number of strands=3
  ]
\braid (1) at (0,0) s_1 1 ;
\node at ([yshift=-1cm]1-rev-2-e) {$\sigma_1$};

\braid (2) at ([xshift=2.5cm] 1-2-s) s_2 1 ;
\node at ([yshift=-1cm]2-rev-2-e) {$\sigma_2$};

\braid (321) at ([xshift=2.5cm] 2-2-s) s_1^{-1} s_2;
\node at ([yshift=-1cm]321-rev-2-e) {\(\sigma_1^{-1} \sigma_2\)};

\braid (123) at ([xshift=2.5cm]321-2-s) s_2^{-1} s_1;
\node at ([yshift=-1cm]123-rev-2-e) {\(\sigma_2^{-1} \sigma_{1}^{}\)};
\end{tikzpicture}
\vspace{-8mm}
    \caption{Geometric $3$-braids}
    \label{fig:3braids}
\end{figure}
The main theorem is as follows.
\begin{theorem}[Main Theorem $1$]
\label{theorem:main}
    Let $B_n$ be the braid group of $n$ strands and set
    \[
    A=\{(x,y) \in \R / \Z \times\mathbb{R}  \mid a \le y \le b \}.
    \]
    For any braid $b \in {B}_{n}$, there exists a Hamiltonian function $H_b=H_b(t,x,y) \in C^{\infty}(\R / \Z \times  A)$ satisfying the following properties:
    there are $n$ points on the plane, say $a_1, \cdots,a_n$, such that the set defined by 
    \[\{ \{\varphi_t(a_1),\cdots,\varphi_t(a_n)\} \times \{t\} \mid t \in [0,1]\}\] 
    represents a geometric braid $b$, where $\varphi_t$ is a Hamiltonian flow for $H_b$.
\end{theorem}

The relationship between braid groups and mapping class groups is well known.
We give a brief remark. 
Let $D_n$ be a $2$-dimensional disk with $n$ punctures, and $\mathrm{MCG}(D_n)$ be the mapping class group of $D_n$.
Then we can construct a surjective homomorphism from $B_{n}$ to $\mathrm{MCG}(D_n)$.
The Nielsen-Thurston classification \cite{Thurston88} implies that each element $g \in \mathrm{MCG}(D_n)$ satisfies one of three conditions:
\begin{enumerate}[(i)]
    \item periodic: there exists $k \in \N$ such that $g^k=\mathrm{id}$.
    \item reducible: there exists a finite union of disjoint simple closed curves $C_1, \cdots,C_k$ that $g$ preserves.
    \item pseudo-Anosov: otherwise.
\end{enumerate}
Thus, we can consider each braid  to have any of the properties (i) to (iii).
In the context of dynamical systems, pseudo-Anosov types are most interesting because they imply some hyperbolicity.
More precisely, if $g$ is pseudo-Anosov, there exist a $\mu>1$ called a ``stretch factor'' and a pair of measured transversal foliations $(F^+, \mu^{+}), (F^{-}, \mu^{-})$ that satisfy
\[
g((F^+, \mu^{+})) = (F^+, \lambda\mu^{+}), \ \text{and} \ g((F^{-}, \mu^{-}))=(F^{-}, \lambda^{-1}\mu^{-}).
\]
For a braid $b \in B_n$ of the pseudo-Anosov type, the Hamiltonian $H_{b}$ obtained in Theorem \ref{theorem:main} is pseudo-Anosov, which immediately indicates the following corollary.
\begin{corollary}[Main Theorem $2$]
For any pseudo-Anosov map derived from a braid, we construct a Hamiltonian system whose Poincar\'e map is isotope to the pseudo-Anosov map.
\end{corollary}

\section{Construction of generators for the geometric braids}
\label{sec:generator}
\subsection{Construction of basic twist maps}
\label{sec:const_sigma}
Firstly, we present a Hamiltonian for trivial braids.
It is simple, and there is no need to use Theorem \ref{theorem:Moser}.
Set
\[
H_{0}(t,x,y) = cy^2.
\]
Taking the initial values $(a_0,0), \cdots, (a_n,0)$, each trajectory draws a straight line, which represents a trivial $n$-braid, and $H_0$ gives trivial braids with $n$ strands.

Next, we construct a Hamiltonian for $\sigma_1$.
Fix $\varepsilon>0$ arbitrarily and take sufficiently small $\delta>0$.
Consider a mapping
\[
\psi \colon (x,y) \mapsto (X,Y)
\]
given by
\begin{align}
\label{desired_twist}
\begin{split}
&(X,Y)=
\begin{cases}
    ( \psi_1(x,y) , \psi_2(x,y)) \ &  \|(x-\lfloor{x}\rfloor, y)\| \le   {\delta} \\
     (x + (\sin \theta) y,y)\ & \|(x-\lfloor{x} \rfloor,y  \rfloor)\| \ge {2\delta}\\
    \tilde\psi(x,y) & \text{otherwise}
\end{cases}
\end{split}
\end{align}
where
\begin{align*}
    \psi_1(x,y)&= (\cos \theta)(x-\lfloor{x} \rfloor)+(\sin \theta) y + \lfloor{x} \rfloor,\\
   \psi_2(x,y)&=- (\sin \theta) (x-\lfloor{x}\rfloor)+(\cos \theta) y,
\end{align*}
and $\tilde\psi(x,y)$ is provided later.
To be precise, we do not define $\tilde\psi$ specifically, but rather define a generating function that gives $\psi$.
To obtain the desired Hamiltonian for a given braid, our strategy is to construct ${\psi}$ such that the map defined above is a smooth monotone twist mapping.
A trajectory near the origin determined by the precise iteration of $\psi$ corresponds to a generator for braids.
For example, set
\begin{align}
\label{eq:a123_set}
a_1=(-\delta,0), a_2=(\delta,0), \ \text{and} \ a_3=(3\delta,0).
\end{align}
Then we can imagine that three points $\{a_1,a_2,a_3\}$ drew $\sigma_1$ in ${B}_3$ by acting the iteration of $\psi$.
It is easily seen that the generating functions for
\begin{align}
\label{eq:map1}
(X,Y)= ((\cos \theta)x +(\sin \theta) y, - (\sin \theta) x+(\cos \theta) y) ,
\end{align}
and
\begin{align}
\label{eq:map2}
(X,Y)= (x + (\sin \theta)y,y)
\end{align}
are respectively denoted as  
\begin{align}
\label{eq:g1}
\begin{split}
g(x,X)
&= \frac{1}{\sin \theta} \left( \frac{\cos \theta}{2} (x^2 + X^2) -x X \right)\\
&=\frac{1}{\sin \theta} \left( \frac{1}{2}(x - X)^2 - \frac{1-\cos \theta}{2} (x^2+X^2) \right)
\end{split}
\end{align}
and 
\begin{align}
\label{eq:g2}
g(x,X)=
\frac{1}{2\sin \theta}(x-X)^2.
\end{align}
The maps $\eqref{eq:map1}$ and $\eqref{eq:map2}$ are determined, respectively, from $\eqref{eq:g1}$ and $\eqref{eq:g2}$ by using the relation
\[
g_x=-y, \ g_X=Y.
\]
We next construct a generating function for $\eqref{desired_twist}$.
Let 
\[ 
g(x, X)= \frac{1}{2\sin \theta} (x-X)^2-\frac{1}{2} \sum_{k \in \mathbb{Z}} \gamma(x-k, X-k)((x-k)^2+(X-k)^2),
\] 
where $\gamma(x, X)$ is smooth, and set
\begin{align}
\label{eq:theta}
    \theta = \frac{\pi}{2n}, \text{or} \ \theta = -\frac{\pi}{2n}
\end{align}
for some $n \in \N \backslash \{1,2\}$, and
\[
\xi =
\frac{1-\cos \theta}{\sin \theta}=\tan \frac{\theta}{2}
\]
This function satisfies $g(x+1, X+1)=g(x, X)$ for any $(x,X) \in \R^2$.
We have
\begin{align*}
    g_x = \frac{1}{\sin \theta}(x-X)- &\frac{1}{2} \sum_{k \in \Z} \gamma_x(x-k, X-k) \{(x-k)^2+(X-k)^2\}\\
    &- \sum_{k \in \Z} \gamma(x-k, X-k)(x-k),
\end{align*}
and
\begin{align*}
    g_{xX} = -\frac{1}{\sin \theta}
    & - \frac{1}{2}\sum_{k \in \Z} \gamma_{xX}(x-k, X-k) \{(x-k)^2+(X-k)^2\}\\
    &- \sum_{k \in \Z} \gamma_{x}(x-k, X-k) (X-k)- \sum_{k \in \Z} \gamma_X(x-k, X-k)(x-k). 
\end{align*}
In order to realize a map satisfying the desired conditions, 
$g_{xX}$ should always be negative
because it is a known fact that if a generating function $g$ for $f$ is $(1,1)$-periodic and there is a positive constant $C$  such that $g_{xX} \le -C$, then $f$ is a monotone twist mapping (See \cite{Bangert88}). 

Near the origin, we get 
\begin{align*}
    g_{xX} = -\frac{1}{\sin \theta} -  \gamma_{xX}(x, X) \frac{x^2+X^2}{2}-  \gamma_{x}(x, X) X  -  \gamma_X(x, X)x
\end{align*}
If $g$ depends only on the distance from the origin, that is, it can be expressed as
\[
{\gamma}(x, X) = \rho(x^2 + X^2)
\]
by abuse of notation,
we can represent $g_{xX}$ by
\begin{align*}
    g_{xX}
    &= -\frac{1}{\sin \theta} - 4xX \rho'' \frac{x^2+X^2}{2}-  4x X \rho'\\
    &= -\frac{1}{\sin \theta} - 4xX \left(\rho'' \frac{x^2+X^2}{2}-   \rho' \right).
\end{align*}
The twist condition $g_{xX}<0$ is satisfied if and only if
\[
4xX \left(\rho'' \frac{x^2+X^2}{2}-\rho' \right)>-\frac{1}{\sin \theta}.
\]
Changing the coordinate $(x, X)$ to $r(\cos \theta, \sin\theta)$ implies that
\[
4r^2\cos \theta \sin^2 \theta \left(\rho'' \frac{r^2}{2}-\rho' \right)
=2r^2 \sin \theta \sin 2\theta \left( \rho'' \frac{r^2}{2} - \rho' \right) >-1.
\]
Hence, to construct $g$ with $g_{xX}<0$, it suffices to define $\rho$ so that
\begin{align}
\label{ineq:f_condition}
\left|\rho'' \frac{r^2}{2}-\rho' \right|<\frac{1}{2r^2}.
\end{align}
Now, let us discuss the construction of $\rho$.
Set two functions $a$ and $b$ by
\begin{align*}
    a(t)=\begin{cases}
    e^{-1/t} & t >0 \\
    0 & t \le 0 
    \end{cases} \ , \ \text{and} \
    b(t)=\frac{a(t)}{a(t)+a(1-t)}.
\end{align*}
Moreover, we define a function $c_{\varepsilon}$ for $\varepsilon>0$ by
\[
c_\varepsilon(t)= b \left( -\frac{t}{\varepsilon} +2 \right)
\]
It is clear that $c_\varepsilon$ satisfies 
\[
c_\varepsilon(t)
\begin{cases}
    =1 & 0 \le t \le \varepsilon \\
    \in (0, 1) & \varepsilon < t < 2 \varepsilon \\
    =0 & t \ge 2\varepsilon
\end{cases}
\]
and, for $t \in (\varepsilon,2 \varepsilon)$, 
\begin{align*}
     c_\varepsilon'(t)=-\frac{1}{\varepsilon}b' \left( -\frac{t}{\varepsilon} +2 \right), \ \text{and} \ 
     c_\varepsilon''(t)=\frac{1}{\varepsilon^2} b'' \left( -\frac{t}{\varepsilon} +2 \right).
\end{align*}
Using $c_{\varepsilon}$, we define $\rho \colon \R_{\ge 0} \to \R$ by
\[
\rho= \xi c_\varepsilon.
\]
By simple calculation, it is easily seen that $f$ satisfies $\eqref{ineq:f_condition}$ if we take small $|\xi|$, that is, $n$ of $\eqref{eq:theta}$ sufficiently large.

In the $\varepsilon$-neighborhood of the origin, $g$ can be represented by
\[
g(x,X)= \frac{(x-X)^2}{2 \sin \theta} - \frac{\xi (x^2+X^2 )}{2}.
\]
Since the relation between $(x,X)$ and $(y,Y)$ is given by
\begin{align*}
y &=-g_x= -\frac{1}{\sin \theta}(x-X)+\xi x
=\left( \xi- \frac{1}{\sin \theta} \right)x + \frac{1}{\sin \theta} X , \ \text{and}\\
Y &= g_X = \frac{1}{\sin \theta} ( X-x) -\xi X
=-\frac{1}{\sin \theta}x + \left( \frac{1}{\sin \theta}-\xi \right)X ,
\end{align*}
we get a map $\psi$ $(x,y) \mapsto (X,Y)$ around the origin defined by
\begin{align}
\label{eq:linear_map}
\begin{split}
X&=(\cos \theta) x + (\sin \theta)y\\
Y&=  -(\sin \theta)x + (\cos \theta) y.
\end{split}
\end{align}
Then the $n$-iteration of $\eqref{eq:linear_map}$ represents the $\pi$- \ (or $(-\pi)$-) rotation.
This is the desired mapping $\varphi (= \psi^n)$.
\begin{remark}
    Although $\psi$ satisfies the twist condition,
    the mapping ${\varphi}$ given above does not.
    In fact, the composition of twist mappings is not, in general, a twist mapping. 
\end{remark}

\subsection{Monotonicity}
In this section, we ensure that the Hamiltonian $H$ obtained from $\varphi$ in the previous step realizes the prescribed braid.

By Theorem \ref{theorem:Moser}, we obtain a Hamiltonian ${H}$ for ${\psi}$, and the flow $(x(t),y(t))$ satisfies
\[
\varphi(x(0),y(0)) = {\psi}^n(x(0),y(0))=(x(n),y(n)).
\]
Clearly, the sets $\eqref{eq:a123_set}$ are transferred to
\[
\varphi(a_1)=(\delta,0), \
\varphi(a_2)=(-\delta, 0), \text{and} \
\varphi(a_3)=(3\delta, 0 ).
\]
Theorem \ref{theorem:Moser} implies
\begin{proposition}
\label{prop:Hamiltonian-psi}
    There exists a Hamiltonian $H$ satisfying the same conditions as Theorem \ref{theorem:Moser} for $\psi$ in Section \ref{sec:const_sigma}.
\end{proposition}
Let $H^{\psi}$ be the corresponding Hamiltonian for $\psi$ obtained from Theorem \ref{theorem:Moser}, and $\psi_t$ be a Hamiltonian flow for $H^{\psi}$.
Then, the set
\begin{align}
\label{eq:3-braid-flow}
 \{ \{\psi_t(a_1),\psi_t(a_2),\psi_t(a_3)\} \times \{t\} \mid t \in [0,n]\}
\end{align}
represents some $3$-braid by the periodicity of $H^{\psi}$.
Based on the above considerations, there is still a possibility that $\eqref{eq:3-braid-flow}$ expresses something other than $\sigma_1$, e.g., $\sigma_1^3$.
Therefore, we need to prove the monotonicity of the given flows.

Along the construction of the flow by Moser \cite{Moser86}, 
the extremals are the straight lines:
\[
x(t) = x_0 + t(x_1-x_0).
\]
Thus $x(t)$ is monotonically decreasing. 
The two solutions, say $(x_1,y_1)$ and $(x_2,y_2)$, with the initial conditions $(x_1(0), y_1(0))=a_1, (x_2(0), y_2(0))=a_2$, construct the braid $\sigma_1$.

\begin{remark}
We provide two remarks.
\begin{enumerate}[(i)]
\item To express $\eqref{eq:3-braid-flow}$ as a $3$-braid, the condition $\psi_s(a_i) \neq \psi_s(a_j)$ for any $s \in [0,n]$ is needed; however, it trivially holds from the existence and uniqueness theorem.
\item Although we omit the details for the calculation of $y(t)$, it is, in fact, also monotone and expressed as a quadratic equation.
\end{enumerate}
\end{remark}

\section{Connection and construction of generators}
\label{sec:connection}
We next construct a mapping that represents $\sigma_i\sigma_j$.
For simplicity, we consider only the case of $(i,j)=(1,2)$.
Let $\varphi_1$ and $\varphi_2$ be the mappings given in Section \ref{sec:const_sigma}, and
each $\varphi_i$ represents $\sigma_i$ for $i=1,2$.
Let $H_i$ be the corresponding Hamiltonian for $\varphi_i$.
Set
\begin{align}
\label{eq:seed-mtm}
\varphi_t (a) =
\begin{cases}
    \varphi_1(t,a)  & t \in [0,n_1] \\
    \varphi_2({t-n_1} , \varphi_1(n_1,a)) & t \in (n_1,n_1+n_2].
\end{cases}
\end{align}
It is clear that $\varphi_t (a) $ is not generally smooth.
Our aim in this section is to use some deformation functions $h_0$, $h_1$, and $h_2$ to construct a ``smooth" Hamiltonian $H$ satisfying $H(t,x,y)=H(t+n_1+n_2,x,y)$ and
\[
H(t,x,y) = 
\begin{cases}
    h_0(t,x,y)  & ( 0 \le t \le \varepsilon)\\
    H_1(n_1(n_1-2\varepsilon)^{-1}(t-\varepsilon),x,y) & (\varepsilon < t < n_1-\varepsilon) \\
    h_1(t,x,y) & ( n_1-\varepsilon \le t \le n_1+\varepsilon)\\
    H_2(n_2(n_2-n_1-2\varepsilon)^{-1}(t-(n_1+\varepsilon)),x,y) & (n_1+\varepsilon < t < n_2-\varepsilon) \\
    h_2(t,x,y)  & ( n_2-\varepsilon \le t \le n_2)
\end{cases},
\]
where $\varepsilon$ is sufficiently small and positive, and $n_i$ is the number of rotations required in Section $\ref{sec:const_sigma}$, that is, $\varphi_i = \tilde{\varphi}^{n_i}$.
According to the smoothing method in \cite{Moser86}, each Hamiltonian $H$ obtained from a monotone twist map is deformed so that it is of class $C^\infty$ as follows: 
\begin{align}
\label{eq:hamiltonian-inf}
\hat{H}(t,x,y)
=
\begin{cases}
    \frac{1}{2} y^2 & (t \in [0,\tilde\varepsilon) \ \text{and} \ t \in (1-\tilde\varepsilon,1))  \\
    \frac{1}{1-2 \tilde\varepsilon} H(\tau,x,y) & (t \in [\tilde\varepsilon,1-\tilde\varepsilon)) 
\end{cases},
\end{align}
where $\tilde\varepsilon$ is positive and sufficiently small, and $\tau = (t-\tilde\varepsilon)/(1-2 \tilde\varepsilon)$.

The Hamiltonian in \eqref{eq:hamiltonian-inf} is $1$-periodic for $t$.
It, however, can be changed to any other period by rescaling.
Thus, we can obtain a new Hamiltonian with $(n_1+n_2)$-period for $t$ so that it connects two Hamiltonian functions obtained from $\varphi_1$ and $\varphi_2$ in $\eqref{eq:seed-mtm}$.
It is seen that such a Hamiltonian is represented by
\[
{H}(t,x,y) = 
\begin{cases}
    \hat{H}_1(t,x,y) & (0 \le t < n_1) \\
    \hat{H}_2(t-n_1,x,y) & (n_1 \le t < n_1+n_2) 
\end{cases}
\]
where $\hat{H}_i$ is $n_i$-periodic for $t$, and it
 is obtained from $\varphi_i$ in $\eqref{eq:seed-mtm}$ by a similar method to $\eqref{eq:hamiltonian-inf}$.
 Rescaling again, we get a Hamiltonian of class $\C^{\infty}(\R \times A)$, that realizes $\sigma_1 \sigma_2$.

\begin{proof}[Proof of Theorem \ref{theorem:main}]
Fix $n \ge 3$ arbitrarily and $b \in B_n$.
We can determine two sequences $\{s_i\}$ and $\{l_i\}$ satisfying
\[
b = \sigma_{s_1}^{l_1} \sigma_{s_2}^{l_2} \cdots \sigma_{s_n}^{l_n}.
\]
By the remarks in Section~\ref{sec:generator} and \ref{sec:connection}, we get a Hamiltonian with period $\sum_{i} |l_i| n_{s_i}$ for $t$
and the desired flow in Theorem \ref{theorem:main} by rescaling.
\end{proof}




\vspace*{33pt}

\noindent
\textbf{Acknowledgement.}~ 
The first author was partially supported by JSPS KAKENHI Grant Number 23K19009 and 23K25778.
The second author was partially supported by JSPS KAKENHI Grant Number 23K25778.

\vspace{11pt}
\noindent
\textbf{Data Availability.}~
Data sharing not applicable to this article as no datasets were generated or analyzed during the current study.


\bibliographystyle{plain}
\bibliography{Braids_Hamiltonian.bib}

\end{document}